\documentclass[12pt]{amsart}
\usepackage{amssymb,amsthm,amsmath,a4}
\usepackage[]{graphicx}
\newtheorem{thm}{Theorem}[section]
\newtheorem{theorem}[thm]{Theorem}
\newtheorem{corollary}[thm]{Corollary}
\newtheorem{lemma}[thm]{Lemma}

\theoremstyle{definition}
\newtheorem{definition}[thm]{Definition}

\newtheorem{problem}[thm]{Problem}

\theoremstyle{remark}
\newtheorem{remark}[thm]{Remark}

\numberwithin{equation}{section}

\newcommand{\al}{\alpha}

\newcommand{\AC}{\mathcal{A}}
\newcommand{\be}{\beta}

\newcommand{\dr}{\mathrm{d}}
\newcommand{\de}{\mathrm{\delta}}

\newcommand{\DC}{\mathcal{D}}

\newcommand{\eps}{\varepsilon}
\newcommand{\f}{\varphi}

\newcommand{\ga}{\gamma}
\newcommand{\Ga}{\Gamma}
\newcommand{\gb}{\mathbf g}

\newcommand{\ka}{\kappa}
\newcommand{\la}{\lambda}

\newcommand{\om}{\omega}
\newcommand{\Om}{\Omega}

\newcommand{\si}{\sigma}

\newcommand{\Ups}{\Upsilon}

\newcommand{\R}{\mathbb{R}}

\newcommand{\im}{\operatorname{Im}}

\def\leq {\leqslant}
\def\geq {\geqslant}

\begin{document}

\title[Pseudodifferential operators ]
{Pseudodifferential operators on manifolds:\\ a coordinate-free approach}

\author{P. McKeag}
\address{Flat 5, 13 Highfield Close, Hither Green, London, UK}
 \email{petermckeag@gmail.com}

\author{Y. Safarov}
\address{Department of Mathematics, King's College London,
Strand, London, UK} \email{yuri.safarov@kcl.ac.uk}

\date{December 2009}



\maketitle

\section*{Introduction}

The theory of pseudodifferential operators (PDOs) is a powerful technique, which has many applications in analysis and mathematical physics. In the framework of this theory, one can effectively construct the inverse of an elliptic differential operator $L$ on a closed manifold, its non-integer powers and even some more general functions of $L$. For operators with constant coefficients in $\R^n$, this can be easily done by applying the Fourier transform. In a sense, the theory of PDOs extends the Fourier transform method to operators with variable coefficients and operators on manifolds at the expense of losing infinitely smooth contributions. This is normally acceptable for theoretical purposes and is useful for numerical analysis, since numerical methods for the determination of the smooth part are usually more stable.

Traditionally, PDOs on manifolds are defined with the use of local coordinates. This leads to certain restrictions on operators under consideration, as all the definitions and results must be invariant with respect to transformations of coordinates. The main aim of this paper is to introduce the reader to a little known approach to the theory of PDOs that allows one to avoid this problem.

The paper is constructed as follows. In Section \ref{S:local} we
recall some basic definitions and results of the classical theory of
PDOs. Their detailed proofs (as well as other relevant statements
and definitions) can be found, for instance, in \cite{H2,Shu,Ta,Tr}.
Section \ref{S:con} gives a brief overview of some elementary
concepts of differential geometry (see  \cite{KN} or any other
textbook for details). In Sections \ref{S:global} and \ref{S:func}
we explain how to define PDOs without using local coordinates and
quote some results from the paper \cite{Sa1} and the conference
article \cite{Sa2}. Section \ref{S:proj} contains new results on
approximate spectral projections of  the Laplacian obtained in the
PhD thesis \cite{McK}. Finally, in Section \ref{S:other} we give a
review of other related results and discuss possible developments in
the field.

Throughout the paper $C_0^\infty$ denotes the space of infinitely
differentiable functions with compact supports, and $\DC'$ is the
dual space of Schwartz distributions. Recall that, by the Schwartz
theorem, for each operator $A:C_0^\infty\mapsto\DC'$ there exists a
distribution $\AC(x,y)\in\DC'$ such that $\langle
Au,v\rangle=\langle \AC(x,y)\,,\,u(y)v(x)\rangle$ for all $u,v\in
C_0^\infty$. The distribution $\AC(x,y)$ is called the {\it Schwartz
kernel} of $A$.

\section{PDOs: local definition and basic properties}\label{S:local}

Let $a(x,y,\xi)$ be a $C^\infty$-function defined on $U\times
U\times\R^n$,  where $U$ is an open subset of $\R^n$.

\begin{definition}\label{def:local-1}
The function $a$ belongs to the class $S^m_{\rho,\de}$
with  $\rho,\de\in[0,1]$ and $m\in\R$ if
\begin{equation}\label{local-0}
\sup_{(x,y)\in K}|\partial_x^\al\partial_y^\beta\partial_\xi^\gamma
a(x,y,\xi)|\ \leq\
C_{K;\alpha,\beta,\gamma}\,(1+|\xi|)^{m+\de(|\al|+|\be|)-\rho|\ga|}
\end{equation}
for each compact set $K\subset U\times U$ and all multi-indices
$\al,\be,\ga$, where $C_{K;\alpha,\beta,\gamma}$ are some positive
constants.
\end{definition}

\begin{definition}\label{def:local-2}
An operator $A:C_0^\infty(U)\mapsto\DC'(U)$ is said to be a {\it
pseudodifferential} operator of class $\Psi^m_{\rho,\de}$ if
\begin{enumerate}
\item[{\bf(c)}]
its Schwartz kernel $\AC(x,y)$ is infinitely differentiable outside the
diagonal $\{x=y\}$,
\item[{\bf(c$_1$)}]
$\AC(x,y)=(2\pi)^{-n}\int e^{i(x-y)\cdot\xi} a(x,y,\xi)\,\dr\xi$ with
some $a\in S^m_{\rho,\de}$ in a neighbourhood of the diagonal.
\end{enumerate}
\end{definition}

The function $a$ in {\bf(c$_1$)} is called an {\it amplitude}, and
the number $m$ is said to be the {\it order} of the amplitude $a$ and
the corresponding PDO $A$. Note that for amplitudes of order $m>-n$
the integral in {\bf(c$_1$)} does not converge in the usual sense.
However, it is well defined as a distribution in $x$ and $y$.

Let $S^{-\infty}:=\bigcap_{m\in\R}S^m_{\rho,\de}$, and let
$\Psi^{-\infty}$ be the class of operators with infinitely
differentiable Schwartz kernels. If $a\in S^{-\infty}$ (that is, if
$a$ and all its derivatives vanish faster than any power of $|\xi|$
as $|\xi|\to\infty$) then the corresponding PDO $A$ belongs to
$\Psi^{-\infty}$. The classical theory of PDOs is used to study
singularities. Therefore one usually assumes that $a$ is defined
modulo $S^{-\infty}$ and that $x$ is close to $y$.

Let $a\in S^m_{\rho,\de}$ and $a_j\in S^{m_j}_{\rho,\de}$ for some
$\rho,\de\in[0,1]$, where $m_j\to-\infty$ as $j\to\infty$. We shall
write
\begin{equation}\label{local-asymp}
a\ \sim\ \sum_ja_j\,,\qquad|\xi|\to\infty\,,
\end{equation}
if $\,a-\sum\limits_{j<k}a_j\in S^{n_k}_{\rho,\de}\,$ where
$n_k\to-\infty$ as $k\to\infty$. Such series $\sum_ja_j$ are called
asymptotic. If $m_j\to-\infty$ then for every collection of
amplitudes $a_j\in S^{m_j}_{\rho,\de}$ there exists an amplitude $a$
satisfying \eqref{local-asymp}. Obviously, if $a'$ is another
amplitude satisfying \eqref{local-asymp} then $a-a'\in S^{-\infty}$
(or, in other words, \eqref{local-asymp} defines $a$ modulo
$S^{-\infty}$).

\begin{lemma}\label{lem:local-1}
Let $z_\tau:=x+\tau (y-x)$ where $\tau\in[0,1]$.
If $\,\de<\rho$ and $a\in S^m_{\rho,\de}$ then
$$
\int e^{i(x-y)\cdot\xi}
a(x,y,\xi)\,\dr\xi\ =\ \int e^{i(x-y)\cdot\xi}
\si_{A,\tau}(z_\tau,\xi)\,\dr\xi
$$
modulo an infinitely differentiable function, where $\si_{A,\tau}(z,\xi)$ is an amplitude of
class $S^m_{\rho,\de}$ given by the asymptotic expansion
\begin{equation}\label{local-1}
\si_{A,\tau}(z,\xi)\ \sim\
\sum_{\al,\be}\frac{(-i)^{|\al|+|\be|}\,\tau^{|\al|}
(1-\tau)^{|\be|}}{\al!\,\be!}\,\left.\partial_x^\al
\partial_y^\be\partial_\xi^{\al+\be}
a(x,y,\xi)\right|_{y=x=z},\qquad|\xi|\to\infty\,.
\end{equation}
\end{lemma}

\begin{proof}[Sketch of proof]
Expand the amplitude $a$  by Taylor's formula at the point
$(x,y)=(z_\tau,z_\tau)$, replace $(y-x)e^{i(x-y)\cdot\xi}$ with
$i\nabla_\xi e^{i(x-y)\cdot\xi}$ and integrate by parts with respect
to $\xi$.
\end{proof}

The amplitude $\si_{A,\tau}(x,\xi)$ is called the {\it
$\tau$-symbol} of the PDO $A$. It is uniquely defined by the operator $A$ modulo
$S^{-\infty}$. The $0$-symbol is usually called just the symbol and
is denoted $\si_A$. The $\frac12$-symbol and $1$-symbol are said to
be the Weyl and the dual symbol respectively.

In the theory of PDOs, properties of operators are usually described
and results are stated in terms of their symbols. The following
composition formula plays a key role in the symbolic calculus.

\begin{theorem}\label{thm:local-1}
Let $A\in\Psi^{m_1}_{\rho,\de}$ and $B\in\Psi^{m_2}_{\rho,\de}$. If
$\de<\rho$ then the composition $AB$ is a PDO of class
$\Psi^{m_1+m_2}_{\rho,\de}$ whose symbol admits the asymptotic
expansion
\begin{equation}\label{local-2}
\si_{AB}(x,\xi)\ \sim\
\sum_\al\frac{(-i)^{|\al|}}{\al!}\,\partial_\xi^\alpha\si_A(x,\xi)\,\partial_x^\alpha
\si_B(x,\xi)\,,\qquad|\xi|\to\infty\,.
\end{equation}
\end{theorem}

\begin{proof}[Sketch of proof]
From the inversion formula for the Fourier transform it follows that
the Schwartz kernel of $AB$ is given by {\bf(c$_1$)} with the
amplitude $a(x,y,\xi)=\si_A(x,\xi)\si_{B,1}(y,\xi)$. Applying Lemma
\ref{lem:local-1} with $\tau=0$ to $a$, we obtain \eqref{local-2}.
\end{proof}

\begin{remark}\label{rem:local-1}
Theorem \ref{thm:local-1} implies, in particular, that the resolvent
of an elliptic differential operator is a PDO. Using \eqref{local-2},
one can also show that a PDO of order $m$ maps $W^s_2\bigcap
C_0^\infty$ into $W_2^{s-m}$, where $W^r_2$ are the Sobolev spaces.
\end{remark}

Note that in the above lemmas the
condition $\de<\rho$ is of crucial importance; if it is not
fulfilled then the terms in the right hand sides of \eqref{local-1}
and \eqref{local-2} do not form asymptotic series.

Clearly, the phase function $(x-y)\cdot\xi$ in {\bf(c$_1$)} depends
on the choice of coordinates on $U$. Passing to new coordinates
$\tilde x$ and $\tilde y$, we obtain
$$
\AC(\tilde x,\tilde y)\ =\ (2\pi)^{-n}\int e^{i(x(\tilde x)-y(\tilde
y))\cdot\xi} a(x(\tilde x),y(\tilde y),\xi)\,\dr\xi\,.
$$
In a sufficiently small neighbourhood of the diagonal $\{\tilde
x=\tilde y\}$, the new phase function $\f(\tilde x,\tilde
y,\xi)=(x(\tilde x)-y(\tilde y))\cdot\xi$ can be written in the form
$$
\f(\tilde x,\tilde y,\xi)\ =\ (\tilde x-\tilde y)\cdot\Phi(\tilde
x,\tilde y)\,\xi,
$$
where $\Phi(\tilde x,\tilde y)$ is a smooth $n\times n$-matrix
function such that $\det\Phi(\tilde x,\tilde y)\ne0$. Changing
variables $\eta=\Phi(\tilde x,\tilde y)\,\xi$, we see that
$$
\AC(\tilde x,\tilde y)\ =\ (2\pi)^{-n}\int e^{i(\tilde x-\tilde
y)\cdot\xi}\tilde a(\tilde x,\tilde y,\eta)\,\dr\eta\,,
$$
where
$$
\tilde a(\tilde x,\tilde y,\eta)\ =\ |\det\Phi(\tilde x,\tilde
y)|^{-1}\,a(x(\tilde x),y(\tilde y),\Phi^{-1}(\tilde x,\tilde
y)\eta)
$$
is a new amplitude. Thus Definition \ref{def:local-2} does not
depend on the choice of coordinates. However, there are two
obvious problems.

\begin{problem}\label{problem-1}
If $a\in S^m_{\rho,\de}$ then, generally speaking, the new amplitude
$\tilde a$ belongs only to the class $S^m_{\rho,\de'}$ with
$\de':=\max\{\de,1-\rho\}$. If $\rho<\frac12$ then $\de'>\rho$ and
the above lemmas fail. Thus for $\de<\rho<\frac12$ it is impossible
to define PDOs of class $S^m_{\rho,\de}$ on a manifold and to develop
a symbolic calculus using local coordinates.
\end{problem}

\begin{problem}\label{problem-2}
If $\max\{\de,1-\rho\}<\rho$ then the ``main part'' of the symbol
$\si_A$ (called the {\it principal} symbol of $A$) behaves as a
function on the cotangent bundle under change of coordinates.
However, lower order terms in \eqref{local-1} do not have an
invariant meaning. Therefore, the coordinate approach does not allow
one to study the subtle properties of PDOs, which depend on the lower
order terms.
\end{problem}

\section{Linear connections}\label{S:con}

The above problems do not arise if we define the phase function
$(x-y)\cdot\xi$ in an invariant way, without using local coordinates.
It is possible, in particular, when the manifold is equipped with a
linear connection. In this section we shall briefly recall some
relevant definitions and results from differential geometry.

Let $M$ be an $n$-dimensional $\,C^\infty$-manifold. Further on we
shall denote the points of $M$ by $x$, $y$ or $z$. The same letters
will be used for local coordinates on $M$. Similarly, $\xi$, $\eta$
and $\zeta$ will denote points of (or the dual coordinates) on the
fibres $T_x^*M$, $T_y^*M$ and $T_z^*M$ of the cotangent bundle
$T^*M$.

We are going to consider operators acting in the spaces of
$\ka$-densities on $M$, $\ka\in\R$. Recall that a complex-valued
``function'' $u$ on $M$ is said to be a $\ka$-density if it behaves
under change of coordinates in the following way
$$
u(y)\ =\ |\det\{\partial x^i/\partial
y^j\}|^\ka\,u\bigl(x(y)\bigr)\,.
$$
The usual functions on $M$ are $0$-densities. The $\ka$-densities
are sections of some complex linear bundle $\Om^\ka$ over $M$. We
denote by $C^\infty(M;\Om^\ka)$ and $C^\infty_0(M;\Om^\ka)$ the
spaces of smooth $\ka$-densities and smooth $\ka$-densities with
compact supports respectively. If $u\in C^\infty_0(M;\Om^\ka)$ and
$v\in C^\infty(M;\Om^{1-\ka})$ then the product $\ u\,v\ $ is a
density and the integral $\,\int_Mu\,v\,dx\,$ is independent of the
choice of coordinates. This allows one to define the inner product
$\,(u,v)=\int_Mu\,\bar v\,dx\,$ on the space of half-densities
$C^\infty_0(M;\Om^{1/2})$ and to introduce the Hilbert space
$L_2(M;\Om^{1/2})$ in the standard way.

In this and the next sections we shall be assuming that the manifold
$M$ is provided with a linear connection $\Ga$ (which may be
non-complete). This means that, for each local coordinate system, we
have fixed a set of smooth ``functions'' $\Ga^i_{jk}(x)$,
$i,j,k=1,\dots,n$, which behave under change of coordinates in the
following way,
\begin{equation}\label{con-1}
\sum_l\,\frac{\partial y^i}{\partial x^l}\,\Ga^l_{pq} (x) \ =\
\sum_{p,q}\,\frac{\partial y^j}{\partial
x^p}\,\Ga^i_{jk}\bigl(y(x)\bigr)\, \frac{\partial y^k}{\partial
x^q}\;+\;\frac{\partial^2 y^i}{\partial x^p\,\partial x^q}\,.
\end{equation}
The ``functions'' $\Ga^i_{jk}(x)$ are called the {\it Christoffel
symbols\/}. They can be chosen in an arbitrary way (provided that
\eqref{con-1} holds), and every set of Christoffel symbols determines
a linear connection of $M$.

A linear connection $\Ga$ is uniquely characterized by the torsion
tensor $T^i_{jk}:= \Ga^i_{jk}-\Ga^i_{kj}$ and the curvature tensor
$$
R^i_{jkl} \ :=\ \partial_{y^k}\Ga^i_{lj} \ -\ \partial_{y^l}\Ga^i_{kj}
\ +\ \sum_p\,\Ga^i_{kp}\Ga^p_{lj}\ -\ \sum_p\,\Ga^i_{lp}\,\Ga^p_{kj} \,.
$$
If both these tensors vanish on an open set $U\subset M$ then one
can choose local coordinates on a neighbourhood of each point $x\in
U$ in such a way that $\Ga^i_{jk}=0$. Such connections are called
{\it flat}. A connection $\Ga$ is called {\it symmetric} if
$T^i_{jk}=0$.

Let $\nu=\sum\nu^k(y)\,\partial_{y^k}$ be a vector field on $M\,$.
The equality \eqref{con-1} implies that
\begin{equation}\label{con-2}
\nabla_\nu\ :=\ \sum_k\,\nu^k(y)\,\partial_{y^k} \ +\
\sum_{i,j,k}\,\Ga^i_{kj}(y)\,\nu^k(y)\eta_i\,\partial_{\eta_j}
\end{equation}
is a correctly defined vector field on $T^*M$. The vector field
\eqref{con-2} is called the {\it horizontal lift} of $\nu\,$. The
horizontal lifts generate a $n$-dimensional subbundle $HT^*M$ of the
tangent bundle $TT^*M$ over $T^*M$, which is called the {\it
horizontal distribution}. The vertical vector fields
$\partial_{\eta_1},\dots,\partial_{\eta_n}$ generate another
$n$-dimensional subbundle $VT^*M\subset TT^*M$ which is called the
{\it vertical distribution}. Since $HT^*M\,\cap\,VT^*M=\{0\}\,$, the
tangent space $T_{(y,\eta)}T^*M$ at each point $(y,\eta)\in T^*M$
coincides with the sum of its horizontal and vertical subspaces.
Obviously, the horizontal subspaces depend on the choice of $\Ga$
whereas the vertical subspaces do not.

A curve in the cotangent bundle $T^*M$ is said to be horizontal (or
vertical) if its tangent vectors belong to $HT^*M$ (or $VT^*M$). For
any given curve $y(t)\subset M$ and covector $\eta_0\in T_{y(0)}^*M$
there exists a unique horizontal curve
$\bigl(y(t),\eta(t)\bigr)\subset T^*M$ starting at the point
$(y(0),\eta_0)\,$. It is defined in local coordinates $y$ by the
equations
$$
\frac d{dt}\,\eta_j(t) \;-\;\sum_{i,k}\,\Ga^i_{kj}
\bigl(y(t)\bigr)\, \dot y^k(t)\,\eta_i(t)\ =\ 0\,,\qquad \forall
j=1,\dots,n\,,
$$
and is called the {\it horizontal lift} of $y(t)\,$. The
corresponding linear transformation $\eta_0\to\eta(t)$ is said to be
the {\it parallel displacement} along the curve $y(t)$. By duality,
horizontal curves and parallel displacements are defined in the
tangent bundle $TM$ (and then in all the tensor bundles over $M$).

A curve $y(t)\subset M$ is said to be a geodesic if the curve
$\bigl(y(t),\dot y(t)\bigr)\subset TM$ is horizontal or,
equivalently, if
$$
{\overset{\cdot\cdot}y}{}^{k}(t)\ +\ \sum_{i,j}\,\Ga^k_{ij}
\bigl(y(t)\bigr)\, {\dot y}^i(t)\,{\dot y}^j(t)\ =\ 0\,,\qquad
\forall k=1,\dots,n\,.
$$
in any local coordinate system. If $U_x$ is a sufficiently small
neighbourhood  of $x$ then for every $y\in U_x$ there exists a
unique geodesic $\ga_{y,x}(t)$ such that $\ga_{y,x}(0)=x$ and
$\ga_{y,x}(1)=y$. The mapping $U_x\ni y\mapsto\dot\ga_{y,x}(0)\in
T_xM$ is a bijection between $U_x$ and a neighbourhood of the origin
in $T_xM$, and the corresponding coordinates on $U_x$ are called the
{\it normal coordinates\/}. In the normal coordinates $y$ centred at
$x$ we have $\ga_{y,x}(t)=x+t(y-x)$, so that $\dot\ga_{y,x}(t)=y-x$
for all $t\in[0,1]$.

Let $\Phi_{y,x}\,:\,T^*_xM\ \to\ T^*_yM$ be the parallel displacement
along the geodesic $\ga_{y,x}$, and let
$\Ups_{y,x}=|\det\Phi_{y,x}|$. One can easily check that $\Ups_{y,x}$
is a density in $y$ and a $(-1)$-density in $x\,$ (the map
$w_x\to\Ups_{y,x}w_x$ is the parallel displacement along $\ga_{y,x}$
between the fibres of the bundle $\Om$). Note that $\Phi_{y,x}$ and
$\Ups_{y,x}$ depend on the torsion tensor, whereas the geodesics are
determined only by the symmetric part of $\Ga$.

Given local coordinates $x=\{x^1,\ldots,x^n\}$, let us denote by
$\nabla_i$ the horizontal lifts of vector fields $\partial_{x^i}$.
For a multi-index $\al$ with $|\al|=q$, let
$\nabla_x^\al=\frac1{q!}\sum\nabla_{i_1}\ldots\nabla_{i_q}$ where
the sum is taken over all ordered collections of indices
$i_1,\ldots,i_q$ corresponding to the multi-index $\al$. The
following simple lemma can be found, for instance, in \cite[Section
3]{Sa2}.

\begin{lemma}\label{lem:taylor}
If $a\in C^\infty(T^*M)$ then $a(y,\Phi_{y,x}\xi)$ admits the following asymptotic expansion,
\begin{equation}\label{taylor}
a(y,\Phi_{y,x}\xi)\ \sim\ \sum_\al\frac1{\al!}\,\dot\ga^\al_{y,x}\,\nabla_x^\al a(x,\xi)\,,
\qquad y\to x\,.
\end{equation}
\end{lemma}

\begin{proof}[Sketch of proof] Write down the left hand side in
normal coordinates $y$ centred at $x$ and apply Taylor's formula.
\end{proof}

\begin{remark}\label{rem:tensor}
If $a$ is a function on $T^*M$ then the  ``hypermatrix''
$\{\partial_\xi^\be\nabla_x^\al
a(x,\xi)\}_{\substack{|\al|=q,|\be|=p}}$ behaves as a $(p,q)$-tensor
under change of coordinates. Therefore all the formulae in the next
section have an invariant meaning and do not depend on the choice of
coordinates.
\end{remark}

\section{PDOs: a coordinate-free approach}\label{S:global}

\begin{definition}\label{def:global-1}
We shall say that an amplitude $a$ defined on $M\times T^*M$ belongs
to the class $S^m_{\rho,\de}(\Ga)$ with $\rho,\de\in[0,1]$ and
$m\in\R$ if
\begin{equation}\label{global-1}
\sup_{(x,z)\in
K}|\partial_x^\al\partial_\zeta^\beta\nabla_{i_1}\nabla_{i_2}\ldots\nabla_{i_q}
a(x,z,\zeta)|\ \leq\
C_{K;\alpha,\beta,q}\,(1+|\xi|)^{m+\de(|\al|+q)-\rho|\be|}
\end{equation}
for each compact set $K\subset M\times M$, all multi-indices
$\al,\be$ and all sets of indices $i_1,\ldots,i_q$, where $\nabla_k$
are horizontal lifts of the vector fields $\partial_{z_k}$ and
$C_{K;\alpha,\beta,\gamma}$ are some positive constants.
\end{definition}

From the definition of the horizontal lifts it follows that $a\in
S^m_{\rho,\de}(\Ga)$ with $\de\geq1-\rho$ if and only if $a$
satisfies \eqref{local-0} in any local coordinate system. In this
case the class $S^m_{\rho,\de}(\Ga)$ is the same for all linear
connections $\Ga$. If $\de<1-\rho$ then $S^m_{\rho,\de}(\Ga)$ depends
on the choice of $\Ga$. Note that \eqref{local-0} is a particular
case of \eqref{global-1}, in which the connection $\Ga$ is flat.

Let us fix a sufficiently small neighbourhood $V$ of the diagonal in
$M\times M\,$ and define $\ z_\tau=z_\tau(x,y)=\ga_{y,x}(\tau)\,$,
where $\tau\in[0,1]\,$ is regarded as a parameter. Consider the
phase function
\begin{equation}\label{global-2}
\f_\tau(x,\zeta,y)\ =\ -\langle\dot\ga_{y,x}(\tau),\zeta\rangle\,,
\qquad (x,y)\in V\,,\quad \zeta\in T^*_{z_\tau}M\,,
\end{equation}
where $\langle\cdot,\cdot\rangle$ denotes the standard pairing
between vectors and covectors. The function $\f_\tau$ is invariantly
defined and, by the above, coincides with $(x-y)\cdot\zeta$ in
normal coordinates $y$ centred at $x$.

\begin{definition}\label{def:global-2}
An operator $A$ acting in the space of $\ka$-densities on $M$ is said
to be a PDO of class $\Psi^m_{\rho,\de}(\Om^\ka,\Ga)$ if
\begin{enumerate}
\item[{\bf(c)}]
its Schwartz kernel $\AC(x,y)$ is infinitely differentiable outside
the diagonal $\{x=y\}$,
\item[{\bf(c$_2$)}]
$\AC(x,y)=(2\pi)^{-n}\,p_{\ka,\tau}(x,y)\int_{T^*_{z_\tau}M}
e^{i\f_\tau(x,\zeta,y)} a(z_s;z_\tau,\zeta)\,\dr\zeta\ $ in a
neighbourhood of the diagonal, where $a\in S^m_{\rho,\de}(\Ga)$,
$p_{\ka,\tau}:=\Ups_{y,z_\tau}^{1-\ka}\Ups_{z_\tau,x}^{-\ka}\,$ and
$s,\tau\in[0,1]$ are some fixed numbers.
\end{enumerate}
\end{definition}

\begin{remark}\label{rem:global-1}
If $y$ are normal coordinates centred at $x$ then
$\,\f_\tau(x,\zeta,y)=(x-y)\cdot\zeta\,$ and the integral
{\bf(c$_2$)} takes the form {\bf(c$_1$)}. However, Definition
\ref{def:local-2} assumes that $x$ and $y$ are the same local
coordinates  on $U$, whereas the above identity holds if we choose
coordinates $y$ depending on the point $x$.
\end{remark}

\begin{remark}\label{rem:global-2}
The weight factor $p_{\ka,\tau}$ is introduced for the following two
reasons.
\begin{enumerate}
\item[(1)]
It makes the definition independent of the choice of coordinates
$\zeta$ in the cotangent space $T^*_{z_\tau}M$.
\item[(2)]
Because of this factor, the Schwartz kernel behaves as a
$(1-\ka)$-density in $y$ and $\ka$-density in $x$, that is,
{\bf(c$_2$)} defines an operator in the space of $\ka$-densities for
all $\ka\in\R$ and all $s,\tau\in[0,1]$. In particular, this allows
us to consider PDOs in the Hilbert space $L_2(M,\Om^{1/2})$ and to
introduce Weyl symbols (corresponding to $\tau=\frac12$).
\end{enumerate}
One can replace $p_{\ka,\tau}$ in Definition \ref{def:global-2} with
any other smooth weight factor $p(x,y)$ which behaves in a similar
way under change of coordinates. The precise choice of the weight
factor seems to be of little importance, since all formulae in the
symbolic calculi corresponding to different weight factors $p$ and
$\tilde p$ can  easily be deduced from each other by expanding the
function $\,p^{-1}\tilde p\,$ into an asymptotic series of the form
\eqref{taylor}, replacing
$\dot\ga_{y,x}(z_\tau)\,e^{i\f_\tau(x,\zeta,y)}$ with
$i\nabla_\zeta\,e^{i\f_\tau(x,\zeta,y)}$ and integrating by parts
with respect to $\zeta$.
\end{remark}

\begin{lemma}\label{lem:global-1}
If $\,\de<\rho$ and $a\in S^m_{\rho,\de}(\Ga)$ then for all $s,\tau\in[0,1]$
$$
p_{\ka,\tau}\int_{T^*_{z_\tau}M}
e^{i\f_\tau(x,\zeta,y)} a(z_s;z_\tau,\zeta)\,\dr\zeta\ =\ p_{\ka,\tau}\int_{T^*_{z_\tau}M}
e^{i\f_\tau(x,\zeta,y)}\si_{A,\tau}(z_\tau,\zeta)\,\dr\zeta
$$
and
$$
p_{\ka,\tau}\int_{T^*_{z_\tau}M}
e^{i\f_\tau(x,\zeta,y)}\si_{A,\tau}(z_\tau,\zeta)\,\dr\zeta\ =\
p_{\ka,s}\int_{T^*_{z_s}M}
e^{i\f_s(x,\zeta,y)}\si_{A,s}(z_s,\zeta)\,\dr\zeta
$$
modulo $C^\infty$-densities, where $\si_{A,\tau}$ and $\si_{A,s}$ are
amplitudes of class $S^m_{\rho,\de}(\Ga)$ given  by the asymptotic
expansions
$$
\si_{A,\tau}(x,\xi)\ \sim\
\sum_\al\frac{(-i)^{|\al|}\,(s-\tau)^{|\al|}}{\al!}\,\left.\partial_\xi^\al
\nabla_y^\al
a(y;x,\xi)\right|_{y=x},\qquad|\xi|\to\infty\,,
$$
$$
\si_{A,s}(x,\xi)\ \sim\
\sum_\al\frac{(-i)^{|\al|}\,(\tau-s)^{|\al|}}{\al!}\
\partial_\xi^\al \nabla_x^\al
\si_{A,\tau}(x,\xi),\qquad|\xi|\to\infty\,.
$$
\end{lemma}

\begin{proof}[Sketch of proof]
The first identity is proved by applying \eqref{taylor} with
$x=z_\tau$ and $y=z_s$ to the function $a(\cdot;z_\tau,\zeta)$ with
fixed $(z_\tau,\zeta)$, substituting
$\dot\ga_{z_s,z_\tau}\,e^{i\f_\tau}=(\tau-s)\,\nabla_\zeta
e^{i\f_\tau}$ and integrating by parts. The second is obtained in a
similar way, after changing variables
$\zeta=\Phi_{z_s,z_\tau}\zeta'$.
\end{proof}

Lemma \ref{lem:global-1} shows that Definition \ref{def:global-2}
does not depend on the choice of $\tau$ and $s$, and that every PDO
$A$ is defined modulo $\Psi^{-\infty}$ by its $\tau$-symbol
$\si_{A,\tau}$. The other way round, for each linear connection
$\Ga$, the $\tau$-symbol $\si_{A,\tau}$ is determined by the
operator $A$ modulo $S^{-\infty}$.

If $A\in\Psi^m_{\rho,\de}(\Om^\ka,\Ga)$ then, in a similar way, one
can show that
\begin{equation}\label{global-adj}
\si_{A^*,\tau}(x,\xi)\ \sim\
\sum_\al\frac{(-i)^{|\al|}\,(1-2\tau)^{|\al|}}{\al!}\
\partial_\xi^\al \nabla_x^\al
\overline{\si_{A,\tau}(x,\xi)},\qquad|\xi|\to\infty\,,
\end{equation}
where $A^*$ is the adjoint operator acting in the space of
$(1-\ka)$-densities. In particular, for the Weyl symbols we have
$\si_{A^*,1/2}-\overline{\si_{A,1/2}}\in S^\infty$ for all
$\ka\in\R$.

\begin{remark}\label{rem:global-3}
The full $\tau$-symbol $\si_{A,\tau}$ depends on $\Ga$ and $\tau$. If
$\max\{\de,1-\rho\}<\rho$ then all the $\tau$-symbols $\si_{A,\tau}$
corresponding to different connections $\Ga$ coincide with the
principal symbol of $A$ modulo a lower order term. However, in the
general case it seems to be impossible to define a principal symbol
of $A$ without introducing an additional structure on the manifold
$M$ or a global phase function (see Subsection \ref{S:other-phase}).
\end{remark}

Let $\Upsilon_\ka(x,y,z):=
\Upsilon^{1-\ka}_{y,z}\,\Upsilon^{2-\ka}_{z,x}\,\Upsilon^{1-\ka}_{x,y}\,$,
$\psi(x,\xi;y,z):=
\langle\dot\ga_{y,x},\xi\rangle-\langle\dot\ga_{z,x},\xi\rangle
-\langle\dot\ga_{y,z},\Phi_{z,x}\xi\rangle\,$ and
$$
P^{(\ka)}_{\be,\ga}(x,\xi)\ =\ \Bigl(
(\partial_y+\partial_z)^\be\partial_y^\ga\,\sum_{|\be'|\le|\be|}\,
\frac 1{\be'!} D^{\be'}_\xi\partial^{\be'}_y(e^{i\psi}
\Upsilon_\ka)\left.\Bigr)\right|_{y=z=x}\,,
$$
where $y$ and $z$ are normal coordinates centred at $x$. The
functions $P^{(\ka)}_{\be,\ga}\in C^\infty(T^*M)$ are polynomials in
$\xi\,$; we shall denote their degrees by $d^{(\ka)}_{\be,\ga}$.

One can easily show that $P^{(\ka)}_{0,\ga}\equiv0$,
$P^{(\ka)}_{\be,0}\equiv0$ and $d^{(\ka)}_{\be,\ga}\le
\min\{\,|\be|,|\ga|\,\}$ for any connection $\Ga$. Moreover, if $\Ga$
is symmetric then
$d^{(\ka)}_{\be,\ga}\le\min\{\,|\be|,|\ga|,(|\be|+|\ga|)/3\,\} $
\cite[Lemma 8.1]{Sa2}.

\begin{theorem}\label{thm:global-1}
Let $\,A\in\Psi^{m_1}_{\rho,\de}(\Om^\ka,\Ga)\,$ and
$\,B\in\Psi^{m_2}_{\rho,\de}(\Om^\ka,\Ga)\,$, where $\rho>\de$.
Assume, in addition, that
\begin{enumerate}
\item[(1)]
either $\rho>1/2\,$,
\item[(2)]
or the connection $\Ga$ is symmetric and $\rho>1/3\,$,
\item[(3)]
or at least one of the PDOs $A$ and $B$ belongs to
$\Psi^m_{1,0}(\Om^\ka,\Ga)$.
\end{enumerate}
Then $AB\in\Psi^{m_1+m_2}_{\rho,\de}(\Om^\ka,\Ga)$ and
\begin{equation}\label{global-3}
\si_{AB}(x,\xi)\ \sim\ \sum_{\al,\be,\ga}\,\frac 1{\al!}\, \frac
1{\be!}\,\frac 1{\ga!}\,\,P^{(\ka)}_{\be,\ga}(x,\xi)\,\,
D_\xi^{\al+\be}\si_A(x,\xi)\,\,D_\xi^\ga\nabla_x^\al\si_B(x,\xi)\,,\qquad|\xi|\to\infty.
\end{equation}
\end{theorem}

\noindent{\it Proof\/} of Theorem \ref{thm:global-1} is similar to
that of Theorem \ref{thm:local-1} (see \cite[Section 8]{Sa2}). In
particular, if the connection $\Ga$ is flat then
$P^{(\ka)}_{\be,\ga}\equiv0$ as $|\be|+|\ga|\ge1\,$ and
\eqref{global-3} turns into \eqref{local-2}.

\begin{remark}\label{rem:global-4}
The conditions on $\rho$ and the estimates for $d^{(\ka)}_{\be,\ga}$
imply that the terms in the right hand side of \eqref{global-3} form
an asymptotic series. It is plausible that the composition formula
\eqref{global-3} holds whenever the orders of the terms in the right
hand side tend to $-\infty$ as $|\al|+|\be|+|\ga|\to\infty$. However,
it is not clear how this can be proved.
\end{remark}

\begin{remark}\label{rem:global-5}
Coefficients of the polynomials $P^{(\ka)}_{\be,\ga}$ are components
of some tensors, which are polynomials in the curvature and torsion
tensors and their symmetric covariant differentials.
\end{remark}

In the same way as in the local theory of PDOs, Theorem
\ref{thm:global-1} implies standard results on the boundedness of
PDOs in the Sobolev spaces and allows one to construct the resolvent
of an elliptic operator in the form of a PDO.

\section{Functions of the Laplacian}\label{S:func}

In this section we assume that $M$ is a compact Riemannian manifold
without boundary and denote
$|\xi|_x:=\sqrt{\sum_{i,j}g^{ij}(x)\,\xi_i\xi_j}$ where $\xi\in
T^*_xM$ and $\{g^{ij}\}$ is the metric tensor. It is well known that
there exist a unique symmetric connection $\Ga_\gb$ on $M$, called
the Levi--Civita connection, such that the function $|\xi|_x$ is
constant along every horizontal curve in $T^*M$.

Denote by $\Delta$ the Laplace operator acting in the space of
half-densities; in local coordinates
$$
\Delta u(x)\ =\
g^{\ka-1}(x)\,\sum_{i,j}\,\partial_{x^i}\bigl(g(x)\,g^{ij}(x)\,
\partial_{x^j}\left(g^{-\ka}(x)u(x)\right)\bigr)\,,
$$
where $g:=|\det g^{ij}|^{-1/2}$ is the canonical Riemannian density.
Let $\nu$ be a self-adjoint first order PDO such that
$-\Delta+\nu>0$, and let $A_\nu:=\sqrt{-\Delta+\nu}$. The operator
$A_\nu$ is a PDO of class $\Psi^1_{1,0}$ whose symbol coincides with
$|\xi|_x$ modulo $S^0$ in any local coordinate system. Thus we have
$A_\nu\in\Psi^1_{1,0}(\Om^{1/2},\Ga)$ for any linear connection
$\Ga$.

\begin{definition}\label{def:func-1}
If $\rho\in(0,1]$, let $S^m_\rho$ be the class of infinitely
differentiable functions $\om$ on $\R$ such that
\begin{equation}\label{func-0}
|\partial_s^j\om(s)|\ \leq\ C_j\,(1+|s|)^{m-j\rho}\,,\qquad\forall
j=0,1,\ldots,
\end{equation}
where $C_k$ are some constants.
\end{definition}

A natural conjecture is that the operator $\om(A_\nu)$ is a PDO
whenever $\om\in S^m_\rho$. If it is true then the symbol of
$\om(A_\nu)$ should coincide with $\om(|\xi|_x)$ modulo lower order
terms. If $\rho<1/2$ then, generally speaking, this function does not
belong to the class $S^m_{\rho,\delta}$ with $\rho>\de$ in any local
coordinate system. However, since its horizontal derivatives
corresponding to the Levi--Civita connection are equal to zero, we
have $\om(|\xi|_x)\in S^m_{\rho,\delta}(\Ga_\gb)$ for all $\om\in
S^m_\rho$. In particular, this implies the following

\begin{lemma}\label{lem:func-1}
Let $\tau\in[0,1)$ and $U_\tau(t):=\exp(itA_\nu^\tau)$. Then
$U_\tau(t)\in\Psi^m_{1-\tau,0}(\Om^{1/2},\Ga_\gb)$ for all $t\in\R$
and $\si_{U_\tau(t)}(x,\xi)=e^{it|\xi|_x^\tau}\,b^{(\tau)}(t,x,\xi)$,
where $b^{(\tau)}\in C^\infty(\R\times T^*M)$ and
$\partial_t^kb^{(\tau)}\in S^0_{1,0}$ for all $k=0,1,\ldots$ and each
fixed $t$.
\end{lemma}

\begin{proof}[Sketch of proof]
Write down $U_\tau(t)$ formally as an integral {\bf(c$_2$)} with an
unknown symbol of the form $e^{it|\xi|_x^\tau}\,b^{(\tau)}(t,x,\xi)$,
substitute the integral into the equation
$\partial_tU_\tau(t)=iA_\nu^\tau U_\tau(t)$, apply the composition
formula \eqref{global-3} to $A_\nu^\tau U_\tau(t)$ and equate terms
of the same order in the right and left hand sides.
\end{proof}

Using Lemma \ref{lem:func-1}, one can construct other functions of
the operator $A_\nu$.

\begin{theorem}\label{thm:func-1}
If $\om\in S^m_\rho$ then
$\om(A_\nu)\in\Psi^m_{\rho,0}(\Om^{1/2},\Ga_\gb)$ and
\begin{equation}\label{func-1}
\si_{\om(A_\nu)}\ \sim\ \om(|\xi|_x)+\sum_{j=1}^\infty
c_{j,\nu}(x,\xi)\,\om^{(j)}(|\xi|_x)\,,\qquad|\xi|\to\infty,
\end{equation}
where $\om^{(j)}:=\partial_s^j\om$ and $c_{j,\nu}(x,\xi)\in
S^0_{1,0}$. The functions $c_{j,\nu}$ are determined recursively by
the equations
\begin{equation}\label{func-2}
\sigma_{A^k_{\nu}}(x,\xi) = |\xi|^k_{x} + \sum_{j=1}^{k}
\frac{k!}{(k-j)!} |\xi|_x^{k-j}\,c_{j,\nu}(x,\xi).
\end{equation}
\end{theorem}

\begin{proof}[Sketch of proof] Define $\om_\tau(s)=\om(s^{1/\tau})$, and let
$\widehat{\om_\tau}(t)$ be the Fourier transform of $\om_\tau$. Then
$$
\omega(A_{\nu}) = (2\pi)^{-1}\int
\widehat{\omega_{\tau}}(t)\,e^{itA_{\nu}^{\tau}} \, \dr t.
$$
Let $\varsigma \in C^\infty_0(\R)$ be equal to $1$ in a
neighbourhood of the origin and have support contained in a small
neighbourhood of the origin. Consider the operators
$$
\omega_{1}(A_{\nu}) =(2\pi)^{-1}\int\varsigma(t)\,
\widehat{\omega_{\tau}}(t)\, e^{itA_{\nu}^{\tau}} \, \dr t,
$$
$$
\omega_{2}(A_{\nu})
=(2\pi)^{-1}\int(1-\varsigma(t))\,\widehat{\omega_{\tau}}(t)\,e^{itA_{\nu}^{\tau}}
\, \dr t.
$$
By integration by parts, the operator $\omega_{2}(A_{\nu})$ can be
written as
$$
\omega_{2}(A_{\nu}) = (2\pi)^{-1}\,A^{-k}_\nu \int D^{k}_{t}
\left((1-\varsigma(t))\,\widehat{\omega_{\tau}}(t)\right)
e^{itA_{\nu}^{\tau}} \, \dr t.
$$
Since $k$ may be chosen arbitrarily large, this shows that
$\omega_{2}(A_{\nu})$ has an infinitely smooth kernel. By Lemma
\ref{lem:func-1}, the operator $\omega_{1}(A_{\nu})$ is a PDO whose
symbol coincides with
$$
(2\pi)^{-1}\int_{-\infty}^{\infty} \varsigma(t)\,
\widehat{\omega_{\tau}}(t)\,e^{it|\xi|_x^\tau}\,b^{(\tau)}(t,x,\xi)\,
\dr t\,.
$$
Expanding $\varsigma(t)b^{(\tau)}(t,x,\xi)$ by Taylor's formula at
$t=0$, we see that the symbol of $\omega_{1}(A_{\nu})$ admits an
asymptotic expansion of the form \eqref{func-1} with some functions
$c_{j,\nu}$. These functions do not depend on $\om$ and can be found
by substituting $\om(s)=s^k$ with $k=1,2,\ldots$ This leads to
\eqref{func-2}.
\end{proof}

\begin{definition}\label{def:func-2}
If $\rho\in(0,1]$, let $S^m_\rho(\gb)$ be the class of
$C^\infty$-functions on $T^*M$ which admit asymptotic expansions of
the form
\begin{equation}\label{func-3}
a(x,\xi)\ \sim\ \sum_{j=0}^\infty
c_j(x,\xi)\,\om_j(|\xi|_x)\,,\qquad|\xi|\to\infty,
\end{equation}
where $c_j\in S^0_{1,0}$, $\om_j\in S^{m_j}_\rho$ with $m_0=m$ and $m_j\to-\infty$. Denote by $\Psi^m_\rho(\Om^{1/2},\gb)$ the class of PDOs acting in the space of half-densities whose $\Ga_\gb$-symbols belong to $S^m_\rho(\gb)$.
\end{definition}

Theorem \ref{thm:func-1} immediately implies that
$\om(A_\nu)\in\Psi^m_\rho(\Om^{1/2},\gb)$ whenever $\om\in
S^m_\rho$. The other way round, any PDOs of class
$\Psi^m_\rho(\Om^{1/2},\gb)$ can be represented in terms of
functions of the operator $A_\nu$.

\begin{lemma}\label{lem:func-2}
For each $A\in\Psi^m_\rho(\Om^{1/2},\gb)$ there exist PDOs
$C_{j,\nu}\in\Psi^0_{1,0}$ and functions $\tilde\om_j\in
S^{l_j}_\rho$ such that $l_0=m$, $l_j\to-\infty$ and
\begin{equation}\label{func-4}
A\ \sim\ \sum_{j=0}^\infty C_{j,\nu}\,\tilde\om_j(A_\nu)\,,
\end{equation}
where the $\,\sim\,$ sign means that the Schwartz kernel of the
difference $A-\sum_{j=0}^kC_{j,\nu}\,\tilde\om_j(A_\nu)$ becomes
smoother and smoother as $k\to\infty$.
\end{lemma}

\begin{proof}[Sketch of proof]
Assume that \eqref{func-3} holds and denote by $C_0$ the PDO with
symbol $c_0(x,\xi)$. Theorems \ref{thm:global-1} and \ref{thm:func-1}
imply that $A=C_0\,\om_0(A_\nu)+A_\nu^{(1)}$ where
$A_\nu^{(1)}\in\Psi^{l_1}_\rho(\Om^{1/2},\gb)$ with
$l_1\leq\max\{m_1,m_0-\rho\}$. The same arguments show that
$A_\nu^{(1)}=C_{1,\nu}\,\tilde\om_1(A_\nu)+A_\nu^{(2)}$ where
$C_{1,\nu}\in\Psi^0_{1,0}$, $\tilde\om_1\in S^{l_1}_\rho$ and
$A_\nu^{(2)}\in \Psi^{l_2}_\rho(\Om^{1/2},\gb)$ where
$l_2\leq\max\{m_2,l_1-\rho\}$. Repeatedly applying this procedure, we
obtain a sequence of operators
$A_\nu^{(k)}\in\Psi^{l_k}_\rho(\Om^{1/2},\gb)$ such that
$A-A_\nu^{(k)}=\sum_{j=0}^{k-1}C_{j,\nu}\,\tilde\om_j(A_\nu)$, where
$C_{j,\nu}$ and $\tilde\om_j$ satisfy the required conditions and
$l_k\to-\infty$ as $k\to\infty$.
\end{proof}

Since $\om_1(A_\nu)\,\om_2(A_\nu)=\om_1\om_2\,(A_\nu)$ for any two
functions $\om_1$ and $\om_2$, combining Theorem \ref{thm:global-1}
and Lemma \ref{lem:func-2}, we obtain

\begin{corollary}\label{cor:func-1}
If $A\in\Psi^{m_1}_\rho(\Om^{1/2},\gb)$ and
$B\in\Psi^{m_2}_\rho(\Om^{1/2},\gb)$ then the composition $AB$ is a
PDO of class $\Psi^{m_1+m_2}_\rho(\Om^{1/2},\gb)$ whose symbol
admits the asymptotic expansion \eqref{global-3}.
\end{corollary}

\begin{remark}\label{rem:func-1}
Under the conditions of Corollary \ref{cor:func-1}, the estimates on
$d^{(\ka)}_{\be,\ga}$ obtained in Section \ref{S:global} do not
directly imply that \eqref{global-3} is an asymptotic series, as it
seems to contain terms of growing orders. However, these ``bad''
terms cancel out due to the symmetries of the curvature tensor. It
would be interesting to find a direct proof of Corollary
\ref{cor:func-1}, which does not use Lemma \ref{lem:func-2} (a
relevant problem was mentioned in Remark \ref{rem:global-4}).
\end{remark}

From the above results it follows that the restriction of the
operator $\om(A_\nu)$ to an open subset of $M$ is determined modulo
$\Psi^{-\infty}$ by the restrictions of the metric $\gb$ and the
operator $\nu$ to this subset. More precisely, we have the following

\begin{corollary}\label{cor:func-2}
Let $\upsilon\in C_0^\infty(M)$, and let $\{\upsilon\}$ be the corresponding
multiplication operator. Consider the operator $\tilde A_{\tilde\nu}$
generated by another metric $\tilde\gb$ and another first order PDO
$\tilde\nu$. If $\tilde\gb=\gb$ on the support of the function $\upsilon$
and $\,\tilde\nu\,\{\upsilon\}=\nu\,\{\upsilon\}\,$ then
$$
\{\upsilon\}\bigl(\om(A_\nu)-\om(\tilde
A_{\tilde\nu})\bigr)\in\Psi^{-\infty} \quad\text{and}\quad
\bigl(\om(A_\nu)-\om(\tilde
A_{\tilde\nu})\bigr)\{\upsilon\}\in\Psi^{-\infty}
$$
for every $\om\in S^m_\rho$.
\end{corollary}

\begin{proof}[Sketch of proof]
The multiplication operator $\{\upsilon\}$ is a PDO with symbol $\upsilon(x)$,
which belongs to $\Psi^0_\rho(\Om^{1/2},\gb)$. Applying Lemma
\ref{lem:func-2} and Corollary \ref{cor:func-1}, we see that
$\,\{\upsilon\}\bigl(\om(A_\nu)-\om(\tilde A_{\tilde\nu})\bigr)\,$ and
$\,\bigl(\om(A_\nu)-\om(\tilde A_{\tilde\nu})\bigr)\{\upsilon\}\,$ are PDOs
whose full symbols are identically equal to zero.
\end{proof}

\begin{remark}\label{rem:func-2}
In a similar way, it is possible to define the classes
$\Psi^m_\rho(\Om^{\ka},\gb)$ which consist of PDOs acting in the
space of $\ka$-densities. Theorem \ref{thm:global-1} implies that
$A\in\Psi^m_\rho(\Om^{\ka},\gb)$ if and only if
$g^{1/2-\ka}Ag^{\ka-1/2}\in\Psi^m_\rho(\Om^{1/2},\gb)$. Using this
observation, one can easily reformulate all results of this section
for operators $A\in\Psi^m_\rho(\Om^{\ka},\gb)$.
\end{remark}

\section{An approximate spectral projection}\label{S:proj}

In applications, one often has to deal with functions of an operator
which depend on additional parameters. It is more or less clear that
the results of the previous section can be extended to
parameter-dependent functions $\om$ under the assumption that the
estimates \eqref{func-0} hold uniformly with respect to the
parameters. Therefore, instead of formulating general statements, we
shall consider an example which is of particular interest for
spectral theory.

Further on we assume that $\la>0$ and denote by $\Psi^{-\infty}(\la)$
the class of parameter-dependent operators with infinitely smooth
Schwartz kernels $\AC_\la(x,y)$ such that
$$
\lim_{\la\to\infty}\la^p\,|\partial_x^\al\partial_y^\be\AC_\la(x,y)|\
=\ 0
$$
for all multi-indices $\al,\be$ and all $p=1,2,\ldots$ Similarly, let
$S^{-\infty}(\la)$ be the class of parameter-dependent amplitudes
$a_\la$ such that
$$
(\la+|\xi|)^p\sup_{(x,y)\in M}|\partial_\xi^\al
\nabla_x^\be\nabla_y^\ga a_\la(y;x,\xi)|\ \to\ 0\quad\text{as}\
\la+|\xi|\to\infty
$$
for all multi-indices $\al,\be,\ga$ and all $p=1,2,\ldots$

Let us fix a small $\eps>0$ and a nonincreasing function $f\in
C^{\infty}(\R)$ such that
$$
f(s) =\left\{ \begin{array}{ll}
1 & \mbox{if $s \leq 0$}; \\
0 & \mbox{if $s \geq \varepsilon$}, \end{array} \right.
$$
and $0\leq f(s)\leq 1$ for all $s\in\R$. If $\rho\in(0,1]$ and
$\la>0$, let
$$
\chi_\rho(\la,s)\ :=\ f(\la^{-\rho}(s-\la))\,.
$$
For each fixed $\la>0$, the function $\chi_\rho$ vanishes on the
interval $[\la+\eps\la^\rho,\infty)$, is identically equal to 1 on
the interval $(-\infty,\la]$ and smoothly descends from 1 to 0 on
the interval $[\la,\la+\eps\la^\rho]$. Since this functions differs
from the characteristic function of the interval $[-\infty,\la]$
only on the relatively small interval $(\la,\la+\eps\la^\rho)$, the
operator $\chi(\la,A_\nu)$ can be thought of as an approximate
spectral projection of $A_\nu$ corresponding to $(-\infty,\la]$. The
standard elliptic regularity theorem implies that the operator
$\chi(\la,A_\nu)$ has an infinitely differentiable Schwartz kernel
for each fixed $\la$.

The derivatives $\partial_s^j\chi_\rho(\la,s)$ are  equal to zero
outside the interval $(\la,\la+\eps\la^\rho)$. Therefore
\begin{equation} \label{proj-1}
|\partial_s^j\chi_\rho(\la,s)|\ \leq\ \tilde
C_j\,(|s|+\la)^{-j\rho}\,,\qquad\forall j=0,1,\ldots,
\end{equation}
for all $s\in\R$ and all $\la>1$, where $\tilde C_j$ are some
constants independent of $\la$ and $s$. The same arguments as in the
proof of Theorem \ref{thm:func-1} show that $\chi(\la,A_\nu)$ is a
parameter-dependent PDO whose symbol admits the asymptotic expansion
\begin{equation}\label{proj-2}
\si_{\chi(\la,A_\nu)}\ \sim\ \chi(\la,|\xi|_x)+\sum_{j=1}^\infty
c_{j,\nu}(x,\xi)\,\chi^{(j)}(\la,|\xi|_x)\,,\qquad\la+|\xi|\to\infty,
\end{equation}
where $c_{j,\nu}(x,\xi)$ are the same functions as in \eqref{func-2}
and $\chi^{(j)}$ denotes $j$th $s$-derivative of the function $\chi$.

Note that the functions $\chi^{(j)}(\la,|\xi|_x)$ belong to
$S^{-\infty}$ for each fixed $\la$. However, their rate of decay
depends on $\la$. The asymptotic expansion \eqref{proj-2} is uniform
with respect to $\la$; it defines $\si_{\chi(\la,A_\nu)}$ modulo
$S^{-\infty}(\la)$. Substituting the terms from \eqref{proj-2} into
the integral {\bf(c$_2$)}, we obtain an asymptotic expansion of the
Schwartz kernel of $\chi(\la,A_\nu)$ into a series of infinitely
smooth half-densities, which decay more and more rapidly as
$\la\to\infty$. This expansion defines $\chi(\la,A_\nu)$ modulo
$\Psi^{-\infty}(\la)$.

Straightforward analysis of the proof of Theorem \ref{thm:global-1}
shows that it remains valid in the case where one of the operators
belongs to $\Psi^0_{1,0}$ and the other is a parameter-dependent PDO
whose symbol admits an asymptotic expansion of the form
\eqref{proj-2}. In this case \eqref{global-3} gives an expansion of
$\si_{AB}$ as $\,\la+|\xi|\to\infty\,$ and defines the symbol modulo
$S^{-\infty}(\la)$.

Now, in the same way as in Section \ref{S:func}, one can show that
the composition of parameter-dependent PDOs whose symbols admit
asymptotic expansions of the form
\begin{equation}\label{proj-3}
\si(x,\xi)\ \sim\ \sum_{j=0}^\infty
c_j(x,\xi)\,\chi^{(j)}(\la,|\xi|_x)\,,\qquad\la+|\xi|\to\infty,\qquad
c_j\in\Psi^0_{1,0}
\end{equation}
is also a parameter-dependent PDO whose symbol is given by
\eqref{global-3} modulo $S^{-\infty}(\la)$.

Let $\Pi_\nu(\la)$ be the spectral projection of the operator
$A_\nu$ corresponding to the interval $(-\infty,\la)$. The above
results imply the following

\begin{theorem}\label{thm:proj-1}
Let $\upsilon\in C_0^\infty(M)$, and let $\{\upsilon\}$ be the corresponding
multiplication operator. Consider the spectral projections
$\Pi_\nu(\la)$ and $\tilde\Pi_{\tilde\nu}(\la)$ generated by
different metrics $\gb$, $\tilde\gb$ and different first order PDOs
$\,\nu$, $\tilde\nu\,$ satisfying the conditions of Section
{\rm\ref{S:func}}. If $\,\tilde\gb=\gb\,$ on the support of the
function $\upsilon$ then
\begin{equation}\label{proj-4}
\Pi_\nu(\la)\,\{\upsilon\}\,\bigl(I-\tilde\Pi_{\tilde\nu}(\la+c\la^\rho)\bigr)\
\in\ \Psi^{-\infty}(\la)\,,\qquad\forall\,c,\rho>0.
\end{equation}
\end{theorem}

\begin{proof}[Sketch of proof]
Assume that $\rho\in(0,1]$, and let $\chi(\la,s)$ be defined as
above with some $\eps<c/3$. Then $\chi(\la,s)\equiv1$ for $s\geq\la$
and $\chi(\la,s-\eps\la^\rho)\equiv0$ for $s\geq\la+c\la^\rho$. It
follows that
$$
\Pi_\nu(\la)\,\chi(\la,A_\nu)=\Pi_\nu(\la)\quad\text{and}\quad
\chi(\la,A_{\tilde\nu}-\eps\la^\rho
I)\,\bigl(I-\Pi_{\tilde\nu}(\la+c\la^\rho)\bigr)=0\,.
$$
Consequently, we have
\begin{multline}\label{proj-5}
\Pi_\nu(\la))\,\{\upsilon\}\,\bigl(I-\Pi_{\tilde\nu}(\la+c\la^\rho)\bigr)\\
=\
\Pi_\nu(\la)\,\chi(\la,A_\nu))\,\{\upsilon\}\,\bigl(I-\chi(\la,A_{\tilde\nu}-\eps\la^\rho
I)\bigr)\,\bigl(I-\Pi_{\tilde\nu}(\la+c\la^\rho)\bigr)
\end{multline}
Since $\chi(\la,s)=\chi(\la,s)\,\chi(\la,s-\eps\la^\rho)$, the
composition formula implies that
$$
\chi(\la,A_\nu))\,\{\upsilon\}\,\bigl(I-\chi(\la,A_{\tilde\nu}-\eps\la^\rho
I)\bigr)\ \in\ \Psi^{-\infty}(\la)\,.
$$
Now the required result follows from \eqref{proj-5} and the fact
that the Schwartz kernel of the spectral projection is polynomially
bounded in $\la$ with all its derivatives (see, for instance,
\cite[Section 1.8]{SV}).
\end{proof}

\begin{remark}\label{rem:proj-1}
It is not surprising that the operator in the left hand side of
\eqref{proj-4} has a lower order than the spectral projections
themselves as $\la\to\infty$. However, one would expect its norm to
decay as a fixed negative power of $\la$, since the perturbation
$A_\nu-A_{\tilde\nu}$ is a more or less arbitrary PDO of order zero.
We do not know whether \eqref{proj-4} can be obtained by other
techniques (including that of Fourier integral operators).
\end{remark}

\begin{remark}\label{rem:proj-2}
All results of this section can easily be extended to a noncompact
closed manifold $M$. In this case all the asymptotic expansions are
unform on compact subsets of $M$ and $M\times M$.
\end{remark}

\section{Other known results and possible developments}\label{S:other}

\subsection{Other definitions for scalar PDOs}\label{S:other-def}
If $\Ga$ is a linear connection, then the corresponding symbol of a
PDO $A$ can easily be recovered from the asymptotic expansion of
$A(e^{i\f_\tau(x,\zeta,y)}\chi(x,y))$ as $\zeta\to\infty$, where
$\f_\tau$ is defined by \eqref{global-2} and $\chi$ is a smooth
cut-off function or $\ka$-density (we suppose that $x$ is fixed and
that the operator acts in the variable $y$). After that, all the
standard formulae of the local theory of PDOs can be rewritten in
terms of their $\Ga$-symbols. Moreover, making appropriate
assumptions about the asymptotic behaviour of
$A(e^{i\f_\tau(x,\zeta,\cdot)}\chi(x,y))$, one can try to define
various classes of PDOs associated with the linear connection $\Ga$.

This approach was introduced and developed by Harold Widom and Lance
Drager (see \cite{Wi1}, \cite{Wi2} and \cite{Dr}). Its main
disadvantage is the absence of an explicit formula representing the
Schwartz kernel of a PDO via its symbol. As a consequence, one has
to assume that PDOs and the corresponding classes of amplitudes are
defined in local coordinates, which makes it impossible to extend
the definition to $\rho<\max\{\de,1-\rho\}$.

In \cite{Pf1}, Markus Pflaum defined a PDO in the space of functions
by the formula
\begin{equation}\label{others-1}
Au(x)\ =\ (2\pi)^{-n}\int_{T^*_xM}\int_{T_xM}\chi(x,y)\,
e^{i\f_0(x,\xi,y)}\,a(x,\xi)\,u(y)\,\dr y\,\dr\xi\,,
\end{equation}
where $a(x,\xi)$ is a function on $T^*M\,$ of class
$S^m_{\rho,\de}\,$, $\,y$ are normal coordinates centred at $x\,$ and
$\chi$ is a smooth cut-off function vanishing outside a neighbourhood
of the diagonal. He obtained asymptotic expansions for the symbols of
the adjoint operator and the composition of PDOs and, in the later
paper \cite{Pf2}, extended them to $\tau$-symbols. However, the
results in \cite{Pf1,Pf2} are stated and proved with the use of local
coordinates and, therefore, the author had to assume that
$\max\{\de,1-\rho\}<\rho$.

Recall that under this condition the standard results of the local
theory of PDOs hold, and the only advantage of a coordinate-free
calculus is that it helps to fight Problem \ref{problem-2}. A typical
example, considered in \cite{Pf1}, is the PDO with a symbol of the
form $(1+|\xi|^2)^{b(x)}$ where $b(x)$ is a smooth function on $M$.
Formally speaking, this PDO belongs only to the class $S^m_{1,\de}$
with $m=\sup_xb(x)$ and any $\de\in(0,1)$. But its properties are
determined by the values of the function $b$ at all points $x\in M$;
in a sense, this operator has a variable order depending on $x\in M$.
In such a situation, it is not sufficient to consider only the
principal symbol. One has to define a full symbol which can be done
with the use of a linear connection.

It is clear that \eqref{others-1} differs from Definition
\ref{def:global-2} only by the choice of the weight factor
$p_{\ka,\tau}$. Applying the procedure described in Remark
\ref{rem:global-2}, one can easily show that
\begin{equation}\label{others-2}
\si_A(x,\xi)\ \sim\ \sum_\al P_\al(x)\,\partial_\xi^\al
a(x,\xi)\,,\qquad|\xi|\to\infty\,,
\end{equation}
where $a(x,\xi)$ is the symbol appearing in \eqref{others-1} and
$P_\al$ are components of some tensor fields. Using
\eqref{others-2}, one can rewrite all the results obtained in
\cite{Sa2} in terms of symbols defined by \eqref{others-1}. This
shows that Pflaum's formulae can be reformulated in terms of the
horizontal derivatives $\nabla_x^\al$ and thus extended to the
classes $\Psi^m_{\rho,\de}(\Om^\ka,\Ga)$ and $\tau$-symbols.

In particular, Pflaum's composition formula can be written in the
form \eqref{global-3} with some  other polynomials $\tilde
P_{\be,\ga}^{(\ka)}$. For operators acting in the space of functions
and $\tau=0$, this result was established by Vladimir Sharafutdinov
in \cite{S1}. He chose to give a direct proof instead of deducing the
formula from \eqref{global-3} and \eqref{others-2} and, for some
reason, considered only the classes $\Psi^m_{1,0}$. Sharafutdinov
gave an alternative description of the polynomials $\tilde
P_{\be,\ga}^{(0)}$ which may be useful for obtaining more explicit
composition formulae (this investigation was continued in \cite{Ga}).
He also proved an analogue of \eqref{global-adj} in the case
$\ka=1/2$ and $\tau=0$ \cite[Theorem 6.1]{S1}.

\begin{remark}\label{rem:others-1}
From \eqref{others-2} it easily follows that the degrees of the
polynomials $\tilde P_{\be,\ga}^{(\ka)}$ admit the same estimates as
$d_{\be,\ga}^{(\ka)}$ (see Section \ref{S:global}).
\end{remark}

\subsection{Operators on sections of vector bundles}\label{S:other-bundles}

In \cite{FK,Pf2,S2,Wi2} the authors considered PDOs acting between
spaces of sections of vector bundles over $M$. In this case, in
order to construct a global symbolic calculus, it is sufficient to
define parallel displacement and horizontal curves in the induced
bundles over $T^*M$. This can be achieved by introducing linear
connections on $M$ and the vector bundles over $M$. After that the
results are stated and proved in the same way as in the scalar case
(further details and references can be found in the above papers).

A more radical approach was proposed by Cyril Levy in \cite{Le}. He
noticed that in order to develop an intrinsic calculus of PDOs one
actually needs only an exponential map, which does not have to be
associated with a linear connection. In his paper Levy assumed that
the manifold $M$ is noncompact and is provided with a global
exponential map (that is, $M$ is a manifold with linearization in the
sense of \cite{Bo}). He then defined associated maps in the induced
vectors bundles and constructed a global coordinate-free symbolic
calculus.

\begin{remark}\label{rem:others-2}
All the papers mentioned in this subsection dealt only with symbols
whose restriction to compact subsets of $M$ belong to
$S^m_{\rho,\de}$ with $\rho>\max\{\de,1-\rho\}$. It should be
possible to extend their results to $\rho<1/2$, using the technique
outlined in Section \ref{S:global}.
\end{remark}

\subsection{Noncompact manifolds}\label{S:other-noncompact}

In order to study global properties of PDOs on a noncompact manifold
$M$, one has to assume that all estimates for symbols and their
derivatives hold uniformly for all $x\in M$ (rather than only on
compact subsets of $M$, as in Definitions \ref{def:local-1} and
\ref{def:global-1}). In \cite{Ba}, Frank Baldus defined classes of
symbols and developed an intrinsic calculus of PDOs on a noncompact
manifold $M$ under the assumption that $M$ has an atlas satisfying
certain global conditions. The statements and proofs in \cite{Ba}
were given in terms of local coordinates, and global results were
obtained by considering the transition maps between coordinates
charts. It is quite possible that these results can be simplified
or/and improved under the assumption that $M$ has a global
exponential map (as in \cite{Le}).

\subsection{Other symbol classes}\label{S:other-symbols}

The paper \cite{Ba} dealt with the more general classes of symbols
$S(m,g)$ instead of $S^m_{\rho,\de}$. The classes $S(m,g)$ were
introduced by L. H\"ormander in \cite{H1} (see also \cite{H2}). They
are defined with the use of coordinates, and in each coordinate
system $S^m_{\rho,\de}$ is a particular case of $S(m,g)$. It would
be interesting to construct similar classes $S(m,g)$ associated with
a linear connection (or an exponential map) and to study the
corresponding classes of symbols and PDOs.

\begin{remark}\label{rem:others-3}
Note that the introduction of ``coordinate'' classes $S(m,g)$ does
not help to resolve Problem \ref{problem-1}. The relation between
these ``coordinate'' classes and the classes $S^m_{\rho,\de}(\Ga)$
was discussed in \cite[Remark 3.5]{Sa2}.
\end{remark}

\subsection{Operators generated by vector fields}\label{S:other-fields}

Let $\nu:=\{\nu_1,\nu_2,\ldots,\nu_n\}$ be a family of smooth vector
fields $\nu_j$ on $M$ which span $T_xM$ at every point $x\in M$.
Consider the corresponding first order differential operators
$\partial_{\nu_j}$ and denote
$$
\partial^\al_\nu:=\frac1{q!}\sum_{j_1,\ldots,j_q}
\partial_{\nu_{j_1}}\partial_{\nu_{j_2}}\ldots\partial_{\nu_{j_q}}
$$
where $q=|\al|$ and the sum is taken over all ordered sets of
indices $j_1,\ldots,j_q$ corresponding to the multi-index
$\al=(\al_1,\ldots,\al_n)$. In other words, $\partial^\al_\nu$ can
be thought of as the symmetrized composition of
$\partial_{\nu_{j_k}}$.

The family $\nu$ generates a unique curvature-free connection
$\Ga_\nu$, with respect to which all covariant derivatives of
the vector fields $\nu_j$ are identically equal to zero. The
$\Ga_\nu$-symbol of $\partial^\al_\nu$ coincides with
$\si_1^{\al_1}\ldots\si_n^{\al_n}$, where
$\si_k=\si_k(x,\xi):=\langle\nu_k,\xi\rangle$ (see \cite[Example
5.4]{Sa2}). Since the functions $\si_k$ are constant along
horizontal curves in $T^*M$ generated by the connection $\Ga_\nu$,
the operators $\partial^\al_\nu$ and their linear combinations can
be regarded as constant coefficient operators relative to the
connection $\Ga_\nu$ (or to the family of the vector fields $\nu$).

This observation was used by Eugene Shargorodsky in \cite{Sha}, where
he developed a complete theory of pseudodifferential operators
generated by a family of vector fields $\nu$. He introduced
anisotropic analogues of classes $S^m_{\rho,\de}$, proved the
composition formula for the corresponding classes of PDOs, defined
semi-elliptic operators associated with the family $\nu$, and
constructed their resolvents. All the results in \cite{Sha} were
obtained for operators acting on sections of vector bundles equipped
with linear connections (see Section \ref{S:other-bundles}).

\subsection{Operators on Lie groups }\label{S:other-lie}

In \cite{RT}, the authors defined full symbols of scalar PDOs on a
compact Lie group $M$ in terms of its irreducible representations
and developed a calculus for such symbols. It would be interesting
to compare their formulae with those obtained by introducing an
invariant linear connection $\Ga$ on $M$ and applying the methods of
\cite{Sa2} or \cite{Sha}.

\subsection{Geometric aspects and physical applications}\label{S:other-physics}

The importance of intrinsic approach in the theory of PDOs for
quantum mechanics is explained in the excellent review \cite{Fu} by
Stephen Fulling. Further discussions can be found in the PhD thesis
\cite{Gu}. Various geometric applications are considered in
\cite{BNPW} and \cite{Vo}. We refer the interested reader to the
above papers and references therein.

\subsection{Global phase functions}\label{S:other-phase}

It is worth noticing that one does not need a linear connection or
even an exponential map to define PDOs on a manifold in a
coordinate-free manner. It is sufficient to fix a globally defined
phase function satisfying certain conditions.

Namely, let $\f(x;y,\eta)$ be an infinitely differentiable function
on $M\times T^*M$ such that
$$
\im\f(x;y,\eta)\geq0\,,\quad\f(x;y,\la\eta)=\la\,\f(x;y,\eta)
$$
for all $x\in M$, $(y,\eta)\in T^*M$ and $\la>0$, and
$$
\f(x;y,\eta)\ =\ (x-y)\cdot\eta\;+\;O(|x-y|^2|\eta|)\,,\qquad x\to y\,,
$$
in any local coordinate system. If $a(x;y,\eta)$ is a smooth function
on $M\times T^*M$ such that $a\in S^m_{1,0}$ in any local coordinate system then
$$
\AC(x,y)\ :=\ \int e^{i\f(x;y,\eta)}\,a(x;y,\eta)\,\dr\eta
$$
is the Schwartz kernel of a PDO $A\in\Psi^m_{1,0}$ acting in the
space of functions. Moreover, there exists an amplitude
$a_\f(y,\eta)$ independent of $x$ such that
$$
\AC(x,y)\;-\;\int e^{i\f(x;y,\eta)}\,a_\f(y,\eta)\,\dr\eta\ \in\
C^\infty(M\times M)\,,
$$
and this amplitude $a_\f$ is uniquely defined by $A$ modulo
$S^{-\infty}$.  The operator $A$ belongs to $\Psi^m_{1,0}$ if and
only if $a_\f\in S^m_{1,0}$ in any local coordinate system.

\begin{remark}\label{rem:others-4}
For a real-valued phase function $\f$ these are standard results of
the theory of Fourier integral operators (see, for instance,
\cite[Section 19]{Shu}). Complex-valued phase functions were
considered in \cite{LSV}.
\end{remark}

It is natural to call $a_\f$ the {\it $\f$-symbol} of the operator
$A$. Clearly, all the standard results of the classical theory of
PDOs can be rewritten in terms of their $\f$-symbols. In particular,
if $A,B\in\Psi^m_{1,0}$ then the $\f$-symbol of the composition $AB$
is determined modulo $S^{-\infty}$ by an asymptotic series which
involves $\f$-symbols of $A$ and $B$ and their derivatives.
Similarly, the $\f$-symbol of the adjoint operator $A^*$ is given by
a series involving the derivatives of $\f$-symbol of $A$.

Obviously, the same formulae remain valid under milder assumptions
about the symbols. Thus it should be possible to introduce symbol
classes associated with the phase function $\f$ and develop a
symbolic calculus in these classes (as was done in \cite{Sa2} for the
special phase function $\f_\tau$ generated by a linear connection).

Such a general approach may allow one to extend results of Section
\ref{S:proj} to other  elliptic operators. It may also be useful for
the study of solutions of hypoelliptic equations and operators on
noncompact manifolds.

\end{document}